\documentclass[a4]{article}%
\usepackage[english]{babel}
\usepackage{amsmath}
\usepackage{amssymb}
\usepackage{amsthm}
\usepackage{amscd}
\usepackage{color}
\usepackage[mathscr]{eucal}
\usepackage{atbegshi}
\AtBeginShipoutFirst{} 
\usepackage[dvipdfmx,colorlinks,linkcolor=blue]{hyperref}
\newtheorem{Thm}{Theorem}[section]
\theoremstyle{definition}
\newtheorem{Theorem}[Thm]{Theorem}
\newtheorem{Lemma}[Thm]{Lemma}
\newtheorem{Corollary}[Thm]{Corollary}
\newtheorem{Proposition}[Thm]{Proposition}
\newtheorem{Definition}[Thm]{Definition}
\newtheorem{Example}[Thm]{Example}
\theoremstyle{remark}
\newtheorem{Remark}{Remark}
\font\sy=cmsy10

\font\ym=msbm10

\newcommand{\N}{{\text{\ym N}}}
\newcommand{\Z}{\text{\ym Z}}

\newcommand{\R}{\text{\ym R}}

\newcommand{\C}{\text{\ym C}}

\newcommand{\cL}{{\hbox{\sy L}}}

\newcommand{\sA}{\mathscr A}
\newcommand{\sB}{\mathscr B}
\newcommand{\sC}{\mathscr C}

\newcommand{\sH}{\mathscr H}
\newcommand{\sK}{\mathscr K}
\newcommand{\sL}{\mathscr L}

\begin{document}
\title
{Notes on Tensor Product Measures}
\author{YAMAGAMI Shigeru\\ Nagoya University\\ Graduate School of Mathematics}
\maketitle   






\section*{Introduction}
Consider a spectral decomposition $U(t) = \int_\R e^{it\tau} E(d\tau)$ of 
a one-parameter unitary group $U(t)$ on a Hilbert space $\sH$,  
where $E(\cdot)$ is a projection-valued measure on $\R$. 
In quantum mechanics, the dynamical behavior of a physical system is described by the associated 
automorphic action of $\R$ on the algebra $\sL(\sH)$ of bounded linear operators. 
The related transition probabilities are associated to 
$(\xi|U(t)TU(t)^*\eta) = (U(t)^*\xi|T U(t)^*\eta)$ 
($\xi, \eta \in \sH$, $T \in \sL(\sH)$), which takes the form 
\[ 
(\int_\R e^{-it\tau} E(d\tau) \xi|T \int_\R e^{-it\tau'} E(d\tau')\eta) 
\]
in terms of the spectral measure. At first glance, it seems quite natural  
to rewrite this to the product measure form like 
\[ 
\int_{\R\times \R} e^{it(\tau - \tau')} (E(d\tau)\xi|T E(d\tau')\eta),  
\] 
which means that we expect a complex-valued measure $(E(d\tau)\xi|T E(d\tau')\eta)$ to be  
well-defined on $\R^2$. This anticipation is reasonably generalized to the following question: 
Let $T \in \sL(\sH)$ and $\xi(\cdot), \eta(\cdot)$ be $\sH$-valued measures 
on a $\sigma$-algebra $\sB$ in a set $S$. 
It is immediate to check that the map $\sB\times \sB \ni A\times B \mapsto (\xi(A)|T\eta(B))$ 
is extended to a finitely additive function $\mu$ on 
the Boolean algebra $\sB\otimes \sB$ generated by $\sB\times\sB$. Is it then possible to 
extend $\mu$ to a complex measure on the $\sigma$-algebra generated by $\sB\times \sB$? 
When $T$ is a finite rank operator, $\mu$ certainly admits such an extension as a linear combination 
of product measures and with a litte more effort we can show that the question is answered affirmatively  
for a trace class operator. 
The general answer, however, turns out to be negative: A bounded linear operator $T$ has 
the measure extension property if and only if 
$T$ is in the Hilbert-Schmidt class ([Swartz1976, Theorem~8]\footnote{We would like to point out, 
however, that the proof there is based on a theorem in another paper, 
which seems difficult to be identified.}
). 

Our main purpose here is to collect relevant results together and combine them to give a self-contained 
proof of it. 

\bigskip
\noindent
Notation: For a Banach space $V$, its unit ball is denoted by $V_1$ and its dual space by $V^*$. 
Given a set $T$, $\ell^\infty(T)$ denotes the Banach space of bounded complex-valued functions on $T$ 
with the sup-norm, which is the dual Banach space of $\ell^1(T)$ of summable functions. 
We then have a canonical isometric embedding 
$V \to \ell^\infty(V_1^*)$ for a Banach space $V$ 
as a restriction of the canonical pairing $V\times V^* \to \C$: 
For $v \in V$, $\widehat{v} \in \ell^\infty(V_1^*)$ is defined by 
$\widehat{v}(v^*) = \langle v, v^*\rangle$ ($v^* \in V_1^*$). 

More generally, if $T \subset V_1^*$ satisfies $\| v\| = \sup\{ |\langle v,v^*\rangle|; v^* \in T\}$, 
then the restriction map $\ell^\infty(V_1^*) \to \ell^\infty(T)$ is isometric on 
$\widehat V = \{ \widehat v; v \in V\}$ and we get an embedding $V \to \ell^\infty(T)$. 

The semi-variation of a finitely additive measure $\phi$ is denoted by $|\phi|$, while 
$\| \phi\|$ is set aside to designate the total semi-variation of $\phi$.

\section{Vector Valued Measures}
We shall mainly deal with Banach spaces as vector spaces and nominate [Diestel-Uhl] as a basic reference. 
See also [Ricker, Chap.1] for a friendly survey on the subject. 

In a (Hausdorff) topological vector space $V$, a family of vectors $\{ v_i\}_{i \in I}$ is said to be 
\textbf{summable} if we can find a vector $v \in V$ fulfilling the following condition: 
Given any neighbourhood $N$ of $v$, we can find a finite subset $F \subset I$ so that 
$\sum_{j \in F \cup F'} v_j \in N$ for any finite subset $F' \subset I \setminus F$. 
The vector $v$ is unique if it exists and denoted as $v = \sum_{i \in I} v_i$. 

When $V$ is a Fr\'echet space, the condition is equivalent to the following: 
Given any neighborhood $N$ of $0$, we can find a finite subset $F$ of $I$ so that 
$\sum_{j \in F'} v_j \in N$ for any finite subset $F' \subset I \setminus F$. 

When $I$ is countable, any counting labeling $\{ i_n; n\geq 1\}$ gives 
\[ 
v = \lim_{n \to \infty} \sum_{k=1}^n v_{i_k}. 
\] 
Conversely, in a Banach space $V$, 
if any counting labeling satisfies the above convergence relation, then $\{ v_i\}$ is summable and 
$v = \sum_{i \in I} v_i$. 
In fact, if not, 
\[ 
\exists \epsilon>0, \forall F \Subset I, \exists F' \Subset I \setminus F, 
\Bigl\| \sum_{j \in F'} v_j \Bigr\| \geq \epsilon 
\]
and we can find a partition $\bigsqcup F_n$ of $I$ by finite subsets satisfying 
$\| \sum_{j \in F_n} v_j \| \geq \epsilon$. Let $\{ i_k\}$ be a counting labeling adapted to 
the increasing sequence $F_1 \subset F_1 \cup F_2 \subset \cdots$. Then $\{ v_{i_k} \}_{k \geq 1}$ 
cannot be a Cauchy sequence by looking at $k = |F_1| + \cdots + |F_n|$ ($n=1,2,\cdots$). 

When $V$ is finite-dimensional with $\| \cdot\|$ any compatible norm, 
the summability of $\{ v_i\}_{i \in I}$ is equivalent to $\sum_{i \in I} \| v_i\| < \infty$, 
the so-called absolute convergence. In fact, for a basis $\{ v_1^*, \cdots, v_n^*\}$, the summability 
implies absolute convergence of $\sum_{i \in I} v_j^*(v_i)$ for $1 \leq j \leq n$, which is equivalent 
to $\sum_{j=1}^n \sum_{i \in I} |v_j^*(v_i)| < \infty$. 
Note that $\sum_{j=1}^n |v_j^*(v)|$ ($v \in V$) defines a norm on $V$. 

\begin{Example}
Let $\{ \delta_n\}_{n \geq 1}$ be an ONB in a separable Hilbert space $\sH$. Then, for $\xi \in \sH$, 
$\{ (\delta_n|\xi) \delta_n\}_{n \geq 1}$ is summable and 
$\xi = \sum_{n \geq 1} (\delta_n|\xi) \delta_n$, 
whereas its absolute convergence is equivalent to the stronger condition 
$\sum_{n \geq 1} |(\delta_n|\xi)| < \infty$. 
\end{Example} 

Let $V$ be a Banach space and $\sB$ be a Boolean algebra in a set $S$. 
A $V$-valued \textbf{semi-measure} is an additive map $\phi: \sB \to V$. 
We say that $\phi$ is \textbf{countably additive} if $\displaystyle A = \bigsqcup_{n=1}^\infty A_n$ is 
a countable partition in $\sB$, then $\displaystyle \phi(A) = \sum_{n \geq 1} \phi(A_n)$. 
By the correspondance between $B_m = \bigsqcup_{n \geq m} A_n = A \setminus (\bigcup_{1 \leq n <m}A_n)$ and 
$A_n = B_n \setminus B_{n+1}$, countable additivity is equivalent to the condition: 
If $B_n \downarrow \emptyset$ in $\sB$, then $\lim_{n \to \infty} \phi(B_n) = 0$. 

A semi-measure $\phi: \sB \to V$ is called a \textbf{measure} if 
$\sB$ is a $\sigma$-algebra and $\phi$ is countably additive. 

Clearly countable additivity implies $\lim_{n \to \infty} \phi(A_n) = 0$. 
We say that a semi-measure $\phi$ is \textbf{squeezing}\footnote{A common terminology for this 
is strong additivity, which is, however, about summability rather than additivity.} if 
$\lim_{n \to \infty} \|\phi(A_n)\| = 0$ for any disjoint sequence $\{ A_n\}_{n \geq 1}$ in $\sB$ 
($\cup_n A_n \in \sB$ being not assumed). 
Remark that a squeezing semi-measure $\phi$ is continuous, i.e.,
$A_n \downarrow \emptyset$ in $\sB$ implies $\lim_n \phi(A_n) = 0$, and, if $\sB$ is a $\sigma$-algebra, 
a continuous semi-measure is squeezing. 
This squeezing property together with finite additivity of $\phi$ 
in turn assures the summability of $\{ \phi(A_n)\}$. 
In fact, non-summability 
\[ 
\exists \epsilon > 0, \forall N, \exists n \geq m \geq N,\  
\Bigl\| \sum_{k=m}^n \phi(A_k) \Bigr\| \geq \epsilon  
\] 
allows us to find a subsequence $1= l_1 < l_2 < \cdots$ satisfying  
$\| \sum_{l_j \leq k < l_{j+1}} \phi(A_k) \| \geq \epsilon$ and we get 
a non-squeezing series $\sum_{j=1}^\infty \phi(B_j)$, where 
$B_j = \cup_{l_j \leq k < l_{j+1}} A_k$ gives a disjoint sequence in $\sB$. 
Notice that $\phi(B_j) = \sum_{l_j \leq k < l_{j+1}} \phi(A_k)$ by finite additivity of $\phi$. 

\begin{Example} 
Let $\sB$ be the power set of $\N$. Then an additive map $\phi: \sB \to V$ gives rise to a sequence 
$\{ v_n = \phi(\{ n\}) \}$ and the squeezing property of $\phi$ implies the Cauchy condition that, 
given $\epsilon > 0$, there exists 
an $N \geq 1$ satisfying $\| \sum_{j \in F} v_j\| \leq \epsilon$ for any 
finite subset $F$ of $\{N+1,N+2, \cdots \}$. 
Conversely given a sequence $\{ v_n\}$ satisfying the Cauchy condition, 
$\{ v_n \}_{n \in A}$ is summable for any subset $A \subset \N$ and 
a countably additive map $\phi: \sB \to V$ is defined by 
\[ 
\phi(A) = \sum_{n \in A} v_n. 
\] 
\end{Example}

\begin{Example}
Let $\sB$ be the Boolean algebra generated by finite subsets of $\N$: 
$A \in \sB$ if and only if either $A$ or $\N \setminus A$ is finite. 
A semi-measure  $\phi: \sB \to \Z$ is then defined by 
\[ 
\phi(A) = 
\begin{cases}
|A| &\text{if $|A| < \infty$,}\\
-|\N \setminus A| &\text{otherwise.}
\end{cases}
\]
\end{Example} 

\begin{Example}
Let $(S,\sB,\mu)$ be a probability space. 
Then $\sB \ni A \mapsto 1_A \in L^p(S,\mu)$ defines a measure for $1 \leq p < \infty$ and 
a semi-measure for $p = \infty$. 
\end{Example}

Let $T: V \to W$ be a bounded linear operator between Banach spaces. Given a semi-measure (resp.~measure) 
$\phi:\sB \to V$, the composite map $T\phi:\sB \to W$ is a semi-measure (resp.~measure). 
As a special case of this, 
we have a semi-measure (resp.~measure) $\widehat\phi: \sB \to \ell^\infty(V_1^*)$ 
as a composition of $\phi$ with the canonical embedding $V \to \ell^\infty(V_1^*)$. 

\begin{Definition}
Given a semi-measure $\lambda: \sB \to \C$, 
the \textbf{variation} of $\lambda$ is a function $|\lambda|: \sB \to [0,\infty]$ defined by 
\[ 
|\lambda|(A) = \sup\{ \sum_{j=1}^n |\lambda(A_j)|; 
\text{$\{ A_j\}$ is a finite partition of $A$ in $\sB$} \}. 
\]
The value $|\lambda|(S)$ is called the total variation of $\lambda$ and denoted by 
$\| \lambda \|$. A semi-measure $\phi$ is said to be of \textbf{bounded variation} 
when $\| \lambda \| < \infty$. 
\end{Definition} 

The following are standard facts on complex (semi-)measures. 

\begin{Proposition}~\label{complex}
\begin{enumerate}
\item
The variation $|\lambda|$ of a complex semi-measure $\lambda$ is additive and satisfies the inequality
\[ 
\sup\{ |\lambda(A)|; A \subset B, A \in \sB\} \leq |\lambda|(B) 
\leq \pi \sup\{ |\lambda(A)|; A \subset B, A \in \sB\}  
\quad 
\text{for $B \in \sB$.} 
\]
\item
The variation of a complex measure $\lambda$ is countably additive and satisfies 
\[ 
|\lambda|(A) = \sup\{ \sum_{j=1}^\infty |\lambda(A_j)|; 
\text{$\{ A_j\}$ is a countable partition of $A$ in $\sB$} \}. 
\]
\item 
Any complex measure defined on a $\sigma$-algebra $\sB$ has a finite total variation and 
the vector space $L^1(\sB)$ 
of all complex measures on $\sB$ is a Banach space with the norm of total variation. 
\end{enumerate}
\end{Proposition}

\begin{Lemma}[Half Average Inequality]
For each positive $d \in \N$, there exists $C_d > 0$ 
($C_1 = 1$, $C_2 = 1/\pi$, $C_3 = 1/4$ and so on) with the following property: 
Given a finite family $\{ v_j \in \R^d \}$ of euclidean vectors, we can find 
a finite subset $J \subset \{1, \cdots, n\}$ so that 
$\sum_{j=1}^n |v_j| \leq |\sum_{j \in J} v_j|/C_d$.  
\end{Lemma} 

\begin{proof}
We may suppose that $v_j \not= 0$. 
For a unit vector $e$, set $(v_j,e)_+ = (v_j,e) \vee 0$, which is a continuous function of $e$. 
In view of the inequality
\[ 
\Bigl| \sum_{(v_j,e) > 0} v_j \Bigr| \geq \sum_{(v_j,e) > 0} (v_j,e) 
= \sum_{j=1}^n (v_j,e)_+, 
\]
let $e_0$ be a unit vector which maximizes the function $\sum_{j=1}^n (v_j,e)$ of $e$ and 
set $J = \{ j; (v_j,e_0) > 0\}$. Then 
$|\sum_{j \in J} v_j| \geq \sum_{j=1}^n (v_j,e_0)_+ \geq \sum_{j=1}^n (v_j,e)_+$ for any $e$ and have  
\[ 
\Bigl| \sum_{j \in J} v_j \Bigr| \geq \sum_{j=1}^n \int_{|e| = 1} (v_j,e)_+\, de 
= C_d \sum_{j=1}^n |v_j| 
\quad 
\text{with} 
\ 
\int_{|e|=1} (v,e)_+\, de = C_d |v| 
\ \text{for any $v$.}
\] 
\end{proof} 

\begin{Definition}
Let $\phi$ be a $V$-valued semi-measure on a Boolean algebra $\sB$ with $V$ a Banach space. 
The \textbf{semi-variation} (\textbf{variation})\footnote{Warning: In literatures, 
semi-variation is denoted by $\|\ \|$, whereas $|\ |$ is used to indicate variation.}
of $\phi$ is a function $|\phi|:\sB \to [0,\infty]$ 
($|\!|\!|\phi|\!|\!|: \sB \to [0,\infty]$) defined by 
\begin{align*} 
|\phi|(A) &= \sup\{ |\langle v^*,\phi\rangle|(A); \| v^*\| \leq 1, v^* \in V^* \},\\ 
|\!|\!|\phi|\!|\!|(A) &= \sup\{ \sum_{j=1}^n \| \phi(A_j) \|; 
\text{$\{ A_j\}$ is a finite partition of $A$ with $A_j \in \sB$} \}. 
\end{align*}
Here $\langle v^*,\phi\rangle$ denotes a complex semi-measure $v^*(\phi(A))$ ($A \in \sB$). 

A semi-measure is said to be \textbf{bounded} (resp.~strongly bounded) 
if $|\phi|$ (resp.~$|\!|\!| \phi|\!|\!|$) is bounded. 
We say that $|\phi|$ is squeezing if $\lim_{n \to \infty} |\phi|(A_n) = 0$ for any disjoint sequence 
$\{ A_n\}_{n \geq 1}$ in $\sB$. 
\end{Definition}

\begin{Proposition}\label{variations}~ 
  \begin{enumerate}
  \item The variation of a $V$-valued measure is a positive measure. 
\item 
A $V$-valued strongly bounded semi-measure $\phi$ defined on a $\sigma$-algebra is 
countably additive if $|\!|\!| \phi|\!|\!|$ is countably additive. 
\item The semi-variation of a $V$-valued semi-measure is monotone, subadditive; 
$|\phi|(A) \leq |\phi|(A \cup B) \leq |\phi|(A) + |\phi|(B)$ for $A, B \in \sB$, and satisfies 
the inequality 
\[ 
\sup \{ \| \phi(A)\|; A \subset B, A \in \sB\} 
\leq |\phi|(B) \leq \pi \sup \{ \| \phi(A)\|; A \subset B, A \in \sB\}, 
\quad 
B \in \sB. 
\] 
Consequently the range of a semi-measure $\phi$ is a bounded subset of $V$ 
if and only if $\phi$ is bounded, i.e., $|\phi|(S) < \infty$. 
\item 
A semi-measure is bounded if it is squeezing. In particular, measures are bounded. 
\item 
A semi-measure $\phi$ is squeezing if and only if so is $|\phi|$. 
  \end{enumerate}
\end{Proposition} 

\begin{proof} (i) $\sim$ (iii) are consequences of definitions by standard arguments. 

(iv): 
Assume that $|\phi|(B) = \infty$ for some $B \in \sB$. 
Then, given any $r>0$, we can find $A \subset B$ in $\sB$ such that 
$\| \phi(A)\| \geq r$ and $\|\phi(B \setminus A)\| \geq r$. In fact, if we choose $A$ so that 
$\| \phi(A)\| \geq r + \| \phi(B)\|$, 
$\| \phi(A) \| = \| \phi(B) - \phi(B \setminus A) \| \leq \| \phi(B)\| + \| \phi(B \setminus A)\|$. 
By a squeezing argument, we obtain a decreasing sequence $\{ A_n\}_{n \geq 1}$ in $\sB$ satisfying 
$|\phi|(A_n) = \infty$ and $\| \phi(A_{n+1})\| \geq 1 + \| \phi(A_n)\|$ for $n \geq 1$. 
Thus, the squeezing property is violated for the disjoint sequence $\{ A_n \setminus A_{n+1} \}_{n \geq 1}$. 

(v): This is a consequence of (iii). The if part is trivial, whereas the only if part is checked as follows: 
If $|\phi|$ is not squeezing, there exist a disjoint sequence $\{ B_n\}$ and $\delta>0$ such that 
$|\phi|(B_n) \geq \delta$ for $n \geq 1$. Then, thanks to the $\pi$-inequality, 
we can find $A_n \subset B_n$ in $\sB$ so that $|\phi|(B_n) \leq \pi \|\phi(A_n)\| + 1/n$, 
which denies $\lim_{n \to \infty} \| \phi(A_n)\| = 0$. 
\end{proof}

\begin{Lemma}\label{variation}
\[ 
|\phi|(A) = \sup\{ \bigl\| \sum \alpha_j \phi(A_j) \bigr\|; 
\text{$\{ A_j\}$ is a finite partition of $A$ with $A_j \in \sB$ and $|\alpha_j| \leq 1$ with 
$\alpha_j \in \C$} \}. 
\]
\end{Lemma} 

\begin{proof}
\[
\Bigl| v^*(\sum \alpha_j \phi(A_j)) \Bigr| \leq \sum |v^*(\phi(A_j))| 
= \sum e^{i\theta_j} v^*(\phi(A_j)) 
= \Bigl| v^*(\sum e^{i\theta_j} \phi(A_j)) \Bigr|. 
\] 
From the first inequality, $|v^*(\sum \alpha_j \phi(A_j))| \leq |v^*\phi|(A)$ and then 
$\| \sum \alpha_j \phi(A_j)\| \leq |\phi|(A)$. From the equalities, 
$\sum |v^*(\phi(A_j))| \leq \| \sum e^{i\theta_j} \phi(A_j)\| \leq \sup \| \sum \alpha_j \phi(A_j)\|$ and then 
$|\phi|(A) \leq \sup \| \sum \alpha_j \phi(A_j)\|$. 
\end{proof} 

\begin{Corollary} Semi-variation remains invariant under taking 
composition with an isometric embedding. 
In particular, $|\widehat\phi| = |\phi|$. Here $\widehat\phi$ denotes 
the composition of $\phi$ with the canonical embedding $V \to \ell^\infty(V^*_1)$. 
\end{Corollary}

Let $\phi:\sB \to V$ be a semi-measure. For a simple function $f:S \to \C$, i.e., 
a function with $f(S)$ a finite set, we 
note that $f = \sum_{z \in f(S)} 1_{[f=z]}$ and set $\phi(f) = \sum_{z \in f(S)} \phi([f=z])$. 
Here $[f=z] = \{ s \in S; f(s) = z\}$. 
By subpartitioning and regrouping, the correspondence $f \mapsto \phi(f)$ is linear and the above lemma 
means $\| \phi\| = |\phi|(S)$. 
Therefore, if $|\phi|(S) < \infty$,  
$\phi$ is continuously extended to the uniform closure $\C(\sB)$ of the set of simple functions. 
Note that $\C(\sB)$ is a commutative C*-algebra. When $\sB$ is a $\sigma$-algebra, 
$\C(\sB)$ is the set of bounded measurable functions and the obvious pairing 
$\C(\sB)\times L^1(\sB) \to \C$ gives rise to 
inclusions $\C(\sB) \subset L^1(\sB)^*$, $L^1(\sB) \subset \C(\sB)^*$. 

Conversely, any bounded linear map $\phi: \C(\sB) \to V$ arises in this way. 
Thus bounded semi-measures form a Banach space with respect to the norm $\| \phi\| = |\phi|(S)$. 

A sequence $\{f_n\}_{n \geq 1}$ of complex-valued functions on $S$ 
is said to \textbf{\boldmath$\mathbf\sigma$-converge} to a function $f$ on $S$ if 
$\{ f_n\}$ is uniformly bounded and $\lim_{n \to \infty} f_n(s) = f(s)$ for every 
$s \in S$. 
Let $\sB^\sigma$ be the $\sigma$-algebra generated by $\sB$. Then $\C(\sB^\sigma)$ is minimal 
among sets which contain $\C(\sB)$ and have the property of being closed under $\sigma$-convergence. 
A $V$-valued measure $\phi$ on a $\sigma$-algebra $\sB$ is $\sigma$-continuous 
in the sense that, if $f_n \in \C(\sB)$ $\sigma$-converges to $f \in \C(\sB)$, 
then $\phi(f_n) \to \phi(f)$ in the weak topology. 


Consider a set $\Lambda$ of complex semi-measures on a Boolean algebra $\sB$ and 
assume that it is bounded in the sense that $\sup \{ |\lambda|(S); \lambda \in \Lambda\} < \infty$. 
We introduce then 
a bounded linear map $\phi_\Lambda: \C(\sB) \to \ell^\infty(\Lambda)$ by 
$\phi_\Lambda(f): \lambda \mapsto \lambda(f) \in \C$. 
From mutual estimates
\[ 
\| \sum \alpha_j \phi_\Lambda(A_j)\| 
= \sup_{\lambda \in \Lambda} |\sum \alpha_j \lambda(A_j)| 
\leq \sup_{\lambda \in \Lambda} |\lambda|(A) 
= \sup_{\lambda \in \Lambda} |\langle \delta_\lambda,\phi_\Lambda\rangle|(A) \leq |\phi_\Lambda|(A), 
\]
we see that $|\phi_\Lambda|(A) = \sup\{ |\lambda|(A); \lambda \in \Lambda\}$ for $A \in \sB$. 
In particular, 
\[ 
\| \phi_\Lambda \| = \sup\{ |\lambda(f)|; \lambda \in \Lambda, f \in \C(\sB)_1\} = 
\sup\{ \|\lambda\|; \lambda \in \Lambda\} < \infty 
\] 
and $|\phi_\Lambda| = |\phi_{|\Lambda|}|$. 
Here we set $|\Lambda| = \{ |\lambda|; \lambda \in \Lambda\}$, which is again a bounded set of 
semi-measures in view of the $\pi$-inequality. 

\begin{Proposition}\label{squeezing}
Consider the following conditions on a bounded set $\Lambda$ of complex semi-measures on a 
Boolean algebra $\sB$. 
\begin{enumerate}
\item $\phi_\Lambda$ is squeezing. 
\item $\phi_\Lambda$ is a measure. 
\item $\phi_{|\Lambda|}$ is squeezing. 
\item $\phi_{|\Lambda|}$ is a measure. 
\end{enumerate}
(i) and (iii) are equivalent. 
If $\sB$ is a $\sigma$-algebra and $\Lambda$ consists of complex measures,  
all the conditions (i) $\sim$ (iv) are equivalent. 
\end{Proposition} 

\begin{proof}
(ii) $\Longrightarrow$ (i) and (iv) $\Longrightarrow$ (iii) are trivial, whereas 
(i) follows from (iii) in view of $|\lambda(A)| \leq |\lambda|(A)$ 
and (iii) $\Longrightarrow$ (iv) is a special case of (i) $\Longrightarrow$ (ii). 

(i) $\Longrightarrow$ (ii): If $\phi_\Lambda$ is not countably additive, we have a disjoint sequence 
$\{ A_n\}$ in $\sB$ and $\delta > 0$ such that $\| \phi_\Lambda(\sqcup_{n \geq m} A_n)\| \geq \delta$ for all $m \geq 1$. 
Then we can inductively find a sequence $\lambda_m \in \Lambda$ and a subsequence $n_1 < n_2 < \cdots$ so that 
$|\lambda_m(\sqcup_{n_m \leq n < n_{m+1}} A_n)| \geq \delta/2$. Now the disjoint sequence $B_m = \cup_{n_m \leq n < n_{m+1}} A_n$ 
satisfies $\| \phi_\Lambda(B_m) \| \geq |\lambda_m(B_m)| \geq \delta/2$ and 
violates the squeezing property of $\phi_\Lambda$. 

(i) $\Longrightarrow$ (iii): If $\phi_{|\Lambda|}$ is not squeezing, we have a disjoint sequence 
$\{ B_n\}$ in $\sB$ and $\delta > 0$ such that $\| \phi_{|\Lambda|}(B_n) \| \geq \delta$. We can therefore find 
a sequence $\lambda_n \in \Lambda$ so that $|\lambda_n|(B_n) \geq \delta/2$ and then, by the $\pi$-inequality, 
a sequence $A_n \subset B_n$ in $\sB$ fulfilling $|\lambda_n|(B_n) \leq \pi |\lambda_n(A_n)| + \delta/3$. 
Now the disjoint sequence $\{ A_n\}$ satisfies $\| \phi_\Lambda(A_n)\| \geq |\lambda_n(A_n)| \geq \delta/6$ and 
denies the squeezing property of $\phi_\Lambda$. 
\end{proof}

\begin{Definition}
Let $\mu$ be a finite positive semi-measure on a Boolean-algebra $\sB$.  
A vector semi-measure $\phi$ on $\sB$ is said to be \textbf{{\boldmath$\mu$}-continuous} if 
$\forall \epsilon > 0, \exists \delta > 0, \forall A \in \sB, 
\mu(A) \leq \delta \Longrightarrow \| \phi(A)\| \leq \epsilon$.  
\end{Definition} 

\begin{Theorem}[Pettis1938] 
Suppose that both of $\phi$ and $\mu$ are measures ($\sB$ being a $\sigma$-algebra necessarily). 
Then $\phi$ is $\mu$-continuous if and only if $\mu(A) = 0$ ($A \in \sB$) implies $\phi(A) = 0$. 
\end{Theorem} 

\begin{proof}
We follow [DU] \S I.2. 
By taking the composition with the canonical embedding $V \to \ell^\infty(V_1^*)$, 
we may suppose that $\phi = \phi_\Lambda$, where $\Lambda = \{ v^*\phi; v^* \in V_1^*\}$ is a bounded 
subset of $L^1(\sB)$. 

If $\phi$ is not $\mu$-continuous, there exists $\delta>0$ and a sequence $A_n \in \sB$ such that 
$\| \phi(A_n)\| \geq \delta$ for $n \geq 1$ and $\sum_{n=1}^\infty \mu(A_n) < \infty$. Let $B_m = \cup_{n \geq m} A_n$ 
be a decreasing sequence in $\sB$. 
From the latter inequality, $\mu(B_m) \downarrow 0$, i.e., 
$\mu(\cap B_m) = 0$. 
From the former inequality, we have  
\begin{align*}
\| \phi_{|\Lambda|}(B_m) \| 
= \sup\{ |v^*\phi|(B_m); v^* \in V_1^*\}  
&\geq \sup\{ |v^*\phi|(A_m); v^* \in V_1^*\}\\
&\geq \sup\{ |\langle v^*,\phi(A_m)\rangle|; v^* \in V_1^*\} = \| \phi(A_m)\| \geq \delta.    
\end{align*}
Since $\phi_\Lambda$ is a measure, $\phi_{|\Lambda|}$ is also a measure by Proposition~\ref{squeezing} 
and the limit $m \to \infty$ is applied to get $\| \phi_{|\Lambda|}(\cap B_m)\| \geq \delta$. 
Therefore we can find a functional $v^* \in V_1^*$ such that $|v^*\phi|(\cap B_m) > \delta/2$ and 
then $A \subset \cap B_m$ in $\sB$ such that $\pi |\langle v^*,\phi(A)\rangle| > \delta/2$. 
Thus $\| \phi(A)\| \not= 0$, whereas $\mu(A) \leq \mu( \cap B_m) = 0$.  
\end{proof} 

\begin{Theorem}[Doubrovsky1947]
Let $\Lambda$ be a bounded set of complex measures on a $\sigma$-algebra $\sB$. 
If $\phi_\Lambda: \sB \to \ell^\infty(\Lambda)$ is a measure, there exists a positive measure 
$\mu \in L^1(\sB)$ for which $\phi_\Lambda$ is $\mu$-continuous with a reverse inequality
$\mu(A) \leq |\phi_\Lambda|(A) = \sup\{ |\lambda|(A); \lambda \in \Lambda\}$. 
\end{Theorem} 

\begin{proof} 
This is \hbox{[DH]}, Theorem I.2.4. 
We first establish a kind of compactness of a bounded $\Lambda$: 
Given any $\epsilon > 0$, we can find a finite subset $F \subset \Lambda$ such that 
if $A \in \sB$ satisfies $|\lambda|(A) = 0$ for $\lambda \in F$, then $|\lambda|(A) \leq \epsilon$ 
for any $\lambda \in \Lambda$. 

If not, there exists $\delta > 0$ such that for any $F \Subset \Lambda$, we can find $A \in \sB$ 
satisfying $|\lambda|(A) = 0$ for any $\lambda \in F$ but $|\nu|(A) \geq \delta$ for some 
$\nu \in \Lambda$. Then we can inductively choose a sequence $\lambda_n$ and $A_n \in \sB$ so that 
$|\lambda_1|(A_n)= \cdots = |\lambda_n|(A_n) = 0$ and $|\lambda_{n+1}|(A_n) \geq \delta$ 
for $n \geq 1$. 
Let $B_m = \cup_{n \geq m} A_n$ be a decreasing sequence and set $B_\infty = \cap_n B_n$. 
Since $|\lambda_m|(B_\infty) \leq |\lambda_m|(B_m) \leq \sum_{n \geq m}|\lambda_m|(A_n) = 0$ for $m \geq 1$ 
and $\lim_{m \to \infty} \sup_{\lambda \in \Lambda} |\lambda|(B_m \setminus B_\infty) = 0$, we have 
\[ 
0 = \lim_{m \to \infty} |\lambda_{m+1}|(B_m \setminus B_\infty) = \lim_{m \to \infty} |\lambda_{m+1}|(B_m), 
\]
which contradicts with $|\lambda_{m+1}|(B_m) \geq |\lambda_{m+1}|(A_m) \geq \delta$.  

Now we use the boundedness of $\Lambda$ again to construct a control measure $\mu$ over $\Lambda$. 
For each $n \geq 1$, choose 
$F_n \Supset \Lambda$ so that $\sum_{\lambda \in F_n} |\lambda|(A) = 0$ with $A \in \sB$ implies 
$|\lambda|(A) \leq 1/n$ for any $\lambda \in \Lambda$. 
Then the positive measure $\mu_n = \sum_{\lambda \in F_n} |\lambda|$ satisfies 
$\mu_n(A) \leq |F_n| |\phi_\Lambda|(A)$ for $A \in \sB$ and, if we define 
\[ 
\mu = \sum_{n=1}^\infty \frac{1}{|F_n| 2^n} \mu_n,  
\]
it is a positive finite measure on $\sB$ with the property $\mu(A) \leq |\phi_\Lambda|(A)$. 
Assume that $\mu(A) = 0$. Then from $\mu_n(A) = 0$, 
$|\lambda|(A) \leq 1/n$ for any $\lambda \in \Lambda$ and any $n \geq 1$, i.e., $|\lambda|(A) = 0$. 
Thanks to the Pettis theorem, this means the $\mu$-continuity of $\phi_\Lambda$. 
\end{proof}

\begin{Corollary}[Bartle-Dunford-Schwartz1955]\label{BDS} 
For a vector measure $\phi$ on a $\sigma$-algebra, we can find a finite positive measure $\mu$ 
so that $\phi$ is $\mu$-continuous and $\mu$ is majorized by $|\phi|$. 
\end{Corollary}

\begin{proof}
Apply the theorem to $\Lambda = \{ v^* \phi; v^* \in V^*_1\}$, 
which is bounded by Proposition~\ref{variations}~(iv). 
\end{proof} 

Recall that countable additivity of a semi-measure $\mu:\sB \to [0,\infty)$ is equivalent to 
the condition that, if $A_n \downarrow \emptyset$ in $\sB$, then $\mu(A_n) \downarrow 0$. 
Let $\sB^\sigma$ be the $\sigma$-algebra generated by $\sB$. 
The classical extension theorem\footnote{This can be attributed to many researchers: 
Fr\'echet, Carath\'eodory, Kolmogorov, Hahn and Hopf (Vladimir I.~Bogachev, Measure Theory I).
It seems fair to add the name of Daniell to the list because 
the theorem itself is just an example of Daniell integral.} says that, 
if a finite positive semi-measure on $\sB$ is countably additive, it is uniquely extended to 
a positive measure on $\sB^\sigma$. 

In the framework of Daniell integral (see \cite{Y2007} for example), this can be explained 
in the following fashion: Let $L$ be the vector lattice 
of real-valued simple functions on the base set $S$. Then a semi-measure $\mu$ can be interpreted 
as a positive linear functional $L \to \R$, which is also denoted by $\mu$. 
Let $f_n \downarrow 0$ in $L$. Then, given $\epsilon>0$, $[f_n \geq \epsilon] \downarrow \emptyset$ 
in $\sB$ and 
$\mu(f_n) \leq \| f_1\|_\infty \mu([f_n \geq \epsilon]) + \epsilon \mu(S)$, 
together with the continuity of $\mu$ imply $\lim_n \mu(f_n) \leq \epsilon \mu(S)$. 
Thus, $\mu$ is continuous as a linear functional and we can apply 
the whole construction of Daniell integral to get a measure extension to $\sB^\sigma$. 

\begin{Lemma}
Let $\mu$ be a finite positive measure on $\sB^\sigma$ and embed $\sB^\sigma$ into $L^1(S,\mu)$. 
Then $\sB^\sigma$ is closed in $L^1(S,\mu)$ and $\sB$ is dense in $\sB^\sigma$. 
\end{Lemma} 

\begin{proof} 
Since any sequential convergence in mean implies almost all convergence by passing to a subsequence, 
$\sB^\sigma$ (more precisely $\{ 1_B; B \in \sB^\sigma\}$) is closed in $L^1(S,\mu)$ 
(pointwise convergence of $\{0,1\}$-valued functions produce $\{ 0,1\}$-valued functions). 
In view of $1_{A \cap B} = 1_A 1_B$ and the dominated convergence theorem, on sees that 
the closure of $\sB$ (more precisely $\{ 1_B; B \in sB\}$) in $L^1(S,\mu)$ provides 
a $\sigma$-algebra and hence coincides with $\sB^\sigma$. 
\end{proof} 

\begin{Theorem}[Kluv\'anek1961]\label{K1961}
Let $\phi: \sB \to V$ be a semi-measure on a Boolean algebra $\sB$.  
If $\phi$ is $\mu$-continuous for some countably additive positive semi-measure $\mu$ on $\sB$, 
then $\phi$ is uniquely extended to a measure $\sB^\sigma \to V$ 
on the $\sigma$-algebra $\sB^\sigma$ generated by $\sB$. 
\end{Theorem} 

\begin{proof} 
The uniqueness is as usual: Given two extensions, sets of their coincidency form a $\sigma$-algebra 
containing $\sB$, whence extensions coincide on the whole $\sB^\sigma$. 

For the existence, first note that $\mu$ is extended to a measure by the classical extension theorem, 
which is again denoted by $\mu$. 
From the previous lemma, a complete (pseudo)metric on $\sB^\sigma$ is defined by 
$d(A,B) = \| 1_A - 1_B\|_1 = \mu(A \triangle B) = \mu(A\setminus B) + \mu(B \setminus A)$ so that 
$(\sB^\sigma,d)$ is the completion of $(\sB,d)$. 
In view of the inequality 
\[ 
\| \phi(A) - \phi(B) \| = \| \phi(A \setminus B) - \phi(B\setminus A)\| 
\leq \| \phi(A \setminus B)\| + \|\phi(B\setminus A)\|, 
\] 
$\phi$ is uniformly continuous with respect to $d$, which admits therefore a continuous extension 
$\overline{\phi}$ to $(\sB^\sigma,d)$. 

Now $\overline{\phi}$ is countably additive thanks to the $d$-continuity: 
finite additivity of $\phi$ goes over to that of $\overline{\phi}$ and monotone convergence assures 
the $\sigma$-additivity. 
\end{proof}

\section{Cross Norms} 
This is a very old but still developing subject and there are lots of references to be mentioned. 
We nominate, however, just [Raymond 1973] and [Diesel 1985] here to follow them. 

Given Banach spaces $X$ and $Y$, let $\sB(X,Y)$ be the Banach space of bounded bilinear forms on 
$X\times Y$ and $\sL(X,Y)$ be the Banach space of bounded linear maps of $X$ into $Y$. 
There are natural identifications $\sL(X,Y^*) = \sB(X,Y) = \sL(Y,X^*)$. 
Recall that a (semi)norm on the algebraic tensor product $X\otimes Y$ is 
called a cross (semi)norm if it satisfies $\| x\otimes y\| = \| x\|\, \|y\|$. 

The obvious bilinear map $X\times Y \to \sB(X^*,Y^*)$ gives rise to a linear map 
$X\otimes Y \to \sB(X^*,Y^*)$ by $(x\otimes y)(x^*,y^*) = x^*(x) y^*(y)$, which is injective. 
In fact, let $z = \sum_l x_l\otimes y_l \in X\otimes Y$ and express 
$z = \sum_{1 \leq j \leq m, 1 \leq k \leq n} z_{j,k}e_j\otimes f_k$ with 
$\{ e_j\} \subset X$ and $\{ f_k\} \subset Y$ linearly independent. 
The dual bases $\{ e_j^*\}$ and $\{ f_k^*\}$ are continuous on finite-dimensional 
subspaces and can be extended to bounded linear functionals. 
We then have $z_{j,k} = z(e_j^*,f_k^*)$. 
Thus the norm on $\sB(X^*,Y^*)$ induces a cross norm on $X\otimes Y$, which is denoted by 
$\| \cdot\|_{\sB(X^*,Y^*)}$. In the embedding $X\otimes Y \to \sB(X^*,Y^*) = \sL(X^*,Y^{**})$, 
$\sum_l x_l\otimes y_l \in X\otimes Y$ is realized by the operator 
$x^* \mapsto \sum_l x^*(x_l) y_l$ and the image of $X\otimes Y$ is included in
the subspace $\sL(X^*,Y) \subset \sL(X^*,Y^{**})$. Thus we have 
\[ 
\| \sum_l x_l\otimes y_l \|_{\sB(X^*,Y^*)}  
= \sup\{ \| \sum_l x^*(x_l) y_l \|; x^* \in X^*_1 \}
= \sup\{ |\sum_l x^*(x_l) y^*(y_l)|; x^* \in X^*_1, y^* \in Y^*_1 \}
\]
and a similar expression holds with the role of $X$ and $Y$ exchanged. 

There is another natural way to get a cross norm on $X\otimes Y$. 
Each $f \in \sB(X,Y)$ defines a linear functional on $X\otimes Y$ by 
$f(z) = \sum_{l=1}^n f(x_l,y_l)$ 
satisfying $|f(z)| \leq \sum_{l=1}^n \| f\| \| x_l\| \| y_l\|$, 
whence it induces a linear map $X\otimes Y \to \sB(X,Y)^*$ and the associated seminorm 
$\|\cdot\|_{\sB(X,Y)^*}$ satisfies 
\[ 
\| z\|_{\sB(X,Y)^*} = \inf \{ \sum_{l=1}^n \| x_l\| \| y_l\|; z = \sum_{l=1}^n x_l\otimes y_l\}. 
\]
The inequality $\leq$ is clear. To get the reverse inequality, we first notice that the right hand side 
defines a seminorm $\| \cdot\|_\text{inf}$ on $X\otimes Y$. Let $\varphi: X \otimes Y \to \C$ be 
a $\|\cdot\|_\text{inf}$-bounded linear functional with its dual norm denoted by $\| \varphi\|$. 
Then the associated bilinear functional $f(x,y) = \varphi(x\otimes y)$ satisfies $\| f\| \leq \|\varphi\|$, 
whence 
\[ 
\|z\| = \sup\{ |\varphi(z)|; \|\varphi\| \leq 1\} 
\leq \sup\{ |f(z)|; f \in \sB(X,Y), \| f\| \leq 1\} = \| z\|_{\sB(X,Y)^*}.
\]

The bilinear map $\Phi: X^*\times Y^* \to \sB(X,Y)$ defined by 
\[ 
\Phi(x^*,y^*): (x,y) \mapsto x^*(x) y^*(y) 
\]
satisfies $\| \Phi(x^*,y^*)\| \leq \| x^*\|\,\|y^*\|$ and it induces a contractive map 
${}^t\Phi: \sB(X,Y)^* \to \sB(X^*,Y^*)$. 
Since the composition of $X\otimes Y \to \sB(X,Y)^*$ with $\sB(X,Y)^* \to \sB(X^*,Y^*)$ 
coincides with the first embedding $X\otimes Y \to \sB(X^*,Y^*)$, we have
$\| z\|_{\sB(X^*,Y^*)} \leq \| z\|_{\sB(X,Y)^*}$. 

Let $X\underline{\otimes} Y$ (resp.~$X\overline{\otimes}Y$) 
be the closure of $X\otimes Y$ in $\sB(X^*,Y^*)$ (resp.~in $\sB(X,Y)^*$). 
Then we have a natural contractive map $X\overline{\otimes} Y \to X\underline{\otimes} Y$. 

These cross norms have the following characterization: 
$\| \cdot\|_{\sB(X,Y)^*}$ is the maximal cross norm, while 
$\| \cdot\|_{\sB(X^*,Y^*)}$ is minimal among cross norms satisfying 
$\| x^*\otimes y^*\| = \| x^*\|\,\|y^*\|$ for $x^*\in X^*$ and $y^* \in Y^*$. 

\begin{Example}
Let $X$ be a Hilbert space. $X\otimes X^* \to \sB(X^*,X^{**}) = \sB(X^*,X) = \sL(X,X)$ is 
an embedding as finite rank operators on $X$ and $X\underline{\otimes} X^*$ corresponds 
to the compact operator algebra $\sC(X)$, 
whereas the norm of $\sB(X,X^*)^* = \sL(X,X^{**})^* = \sL(X,X)^*$ on $X\otimes X^*$ 
is realized by the trace norm on finite rank operators and $X\overline{\otimes}Y$ is identified with 
the trace ideal $\sC_1(X)$ of $\sC(X)$. For $z \in X\otimes X^* \subset \sC_1(X)$, 
\[ 
\| \sum_n z(\delta_n)\otimes \delta_n^* \|_{\sB(X,X^*)^*} 
= \sup \{ |\langle z,\varphi\rangle|; \varphi \in \sL(X), \| \varphi\| \leq 1\} 
= \sup \{ |\langle z,\varphi\rangle|; \varphi \in \sC_1(X)^*, \| \varphi\| \leq 1\} 
= \| z\|_1.  
\] 
\end{Example}

In connection with tensor product measures, 
we introduce two more cross norms $\| \cdot\|_r$ and $\| \cdot\|_l$ according to H.~Jacobs: 
\begin{align*}
\| z\|_r &= \inf\Bigl\{ \sup\{ \bigl\| \sum_i \alpha_i \| x_i\| y_i\bigr\|; |\alpha_i| \leq 1\}; 
z = \sum_{i=1}^n x_i\otimes y_i\Bigr\},\\ 
\| z\|_l &= \inf\Bigl\{ \sup\{ \bigl\| \sum_i \alpha_i \| y_i\| x_i\bigr\|; |\alpha_i| \leq 1\}; 
z = \sum_{i=1}^n x_i\otimes y_i\Bigr\}. 
\end{align*}
It is immediate to show that these are seminorms. 
Clearly these are majorized by the largest cross norm and 
\begin{align*}
\sup\{ \bigl\| \sum_i \alpha_i \| x_i\| y_i\bigr\|; |\alpha_i| \leq 1\}
&= \sup\{ \Bigl| \sum_i \alpha_i \| x_i\| y^*(y_i) \Bigr|; |\alpha_i| \leq 1, \| y^*\| \leq 1 \}\\
&= \sup\{ \sum_i \| x_i\| |y^*(y_i)|; \| y^*\| \leq 1 \}\\
&\geq \sup\{ \| \sum_i y^*(y_i) x_i\|; \| y^*\| \leq 1 \}\\
&= \sup\{ | \sum_i y^*(y_i) x^*(x_i)|; \| x^*\| \leq 1, \| y^*\| \leq 1 \}\\
&= \| \sum_i x_i\otimes y_i\|_{\sB(X^*\times Y^*)}. 
\end{align*}
shows that these majorize the lower cross norm. 
Consequently $\|\cdot\|_l$ and $\|\cdot\|_r$ are in fact cross norms. 
In general, these two norms are different and 
their arithmetic mean gives another cross norm, which is denoted by $\| \cdot \|_m$. 

\begin{Theorem}[Kluv\'anek1973]
Let $\varphi: \sA \to V$ and $\psi: \sB \to W$ be measures and 
$V\otimes_mW$ be the completion of $V\otimes W$ with respect to the cross norm $\| \cdot\|_m$. 
Then there exists a measure $\phi: \sA\otimes_\sigma \sB \to V\otimes_mW$ satisfying 
$\phi(A\times B) = \varphi(A)\otimes \psi(B)$ for $A \in \sA$ and $B \in \sB$. 
\end{Theorem} 

\begin{proof}
Recall that $\varphi$ and $\psi$ are bounded (Proposition~\ref{variations}) and satisfy
\[ 
\| \sum_i \alpha_i \varphi(A_i)\| \leq r |\varphi|(\bigsqcup A_i), 
\quad 
\| \sum_j \beta_j \psi(B_j)\| \leq r |\psi|(\bigsqcup_j B_j) 
\] 
for $\bigsqcup A_i$ in $\sA$ and $\bigsqcup B_j$ in $\sB$  
with $|\alpha_i| \leq r$ and $|\beta_j| \leq r$ (Lemma~\ref{variation}). 

Let $\mu$ and $\nu$ be control measures of $\varphi$ and $\psi$ respectively 
(their existence guaranteed by Corollary~\ref{BDS}). 
Since the map $\sA\times \sB \ni A\times B \mapsto \varphi(A)\otimes_m \psi(B)$ is always 
uniquely extended to a semi-measure $\phi: \sA\otimes \sB \to V\otimes_mW$ on the Boolean algebra 
$\sA\otimes \sB$ generated by $\sA\times \sB$, it suffices to show the $(\mu\times \nu)$-continuity of 
$\phi$ (Theorem~\ref{K1961}). 

So, given $\epsilon > 0$, choose $\delta > 0$ such that 
$\mu(A) \leq \delta$ and $\nu(B) \leq \delta$ imply 
$\| \varphi(A)\| \leq \epsilon$ and $\| \psi(B)\| \leq \epsilon$ for 
$A \in \sA$ and $B \in \sB$. 
Let $C \in \sA\otimes \sB$ satisfy $(\mu\times \nu)(C) \leq \delta^2$. 
If we write $C = \bigsqcup_{i \in I} A_i\times B_i$ with $\bigsqcup A_i$ and set 
$\Delta = \{ i \in I; \nu(B_i) \geq \delta\}$, then the inequalities
$\sum_{k \in \Delta} \delta \mu(A_k)\leq \sum_{k \in \Delta} \mu(A_k)\nu(B_k) \leq (\mu\times\nu)(C) \leq \delta^2$ 
imply $\mu(A_{\Delta}) \leq \delta$ and therefore 
$\| \varphi(A_{\Delta})\| \leq \epsilon$ ($A_J = \sqcup_{j \in J} A_j$ for a subset $J \subset I$). 

Now, in the obvious inequality 
\[ 
\| \sum_k \varphi(A_k)\otimes \psi(B_k)\|_l \leq 
\sup \left\| \sum_k \alpha_k \| \psi(B_k)\| \varphi(A_k) \right\|, 
\] 
if we put $\beta_k = \alpha_k \| \psi(B_k)\|$, 
then $|\beta_k | \leq \epsilon$ for $k \not\in \Delta$ and $|\beta_k| \leq \| \psi\|$ for $k \in \Delta$, 
which are used to get 
\begin{align*} 
\| \phi(C)\|_l &\leq \sup \left\| \sum_{k \in I \setminus \Delta} \alpha_k \| \psi(B_k)\| \varphi(A_k) \right\| 
+ \sup \left\| \sum_{k \in \Delta} \alpha_k \| \psi(B_k)\| \varphi(A_k) \right\|\\ 
&\leq \sup \left\| \sum_{k \in I \setminus \Delta} \beta_k  \varphi(A_k) \right\| 
+ \sup \left\| \sum_{k \in \Delta} \beta_k \varphi(A_k) \right\|\\ 
&\leq \epsilon |\varphi|(A_{I \setminus \Delta}) 
+ \| \psi\|\, |\varphi|( A_\Delta)
\leq \epsilon (\| \varphi\| + \| \psi\|). 
\end{align*} 
By symmetry, we have $\| \phi(C)\|_r \leq \epsilon(\|\varphi\| + \| \psi\|)$ as well and finally get 
$\| \phi(C)\|_m \leq \epsilon(\|\varphi\| + \| \psi\|)$. 
\end{proof} 

\begin{Corollary}[Duchon-Kluv\'anek1967]
Tensor product measures are defined with respect to the least cross norm. 
\end{Corollary} 

To get further information on tensor product measures, we look into cross norms of 
$\ell^p$-sequences in a Banach space $X$. 
Let $\ell_0 \subset \ell^p$ ($1 \leq p < \infty$) 
be a dense subspace consisting of all finite sequences.  
Then $\ell^p\underline{\otimes} X$ contains the algebraic tensor product $\ell_0 \otimes X$ 
as a dense subspace and, for $\sum_n \delta_n\otimes x_n \in \ell_0\otimes X$, 
the lower norm in $\sL(X^*,\ell^p)$ and $\sL(\ell^q,X)$ ($1/p + 1/q = 1$) is evaluated by 
\[ 
\| \sum_n \delta_n\otimes x_n \|_{\sB(\ell^p,X)} = \sup\{ (\sum_n |x^*(x_n)|^p)^{1/p}; x^* \in X^*_1 \} 
= \sup\{ \| \sum_n \lambda_n x_n\|; 
(\lambda_n) \in \ell^q_1 \}.  
\] 
Here is also an intermediate cross norm defined by 
\[ 
\| \sum_n \delta_n\otimes x_n\| = \left( \sum_n \| x_n\|^p\right)^{1/p} 
= \sup\{ \sum_n |\lambda_n|\, \| x_n\|; (\lambda_n) \in \ell^q_1 \}. 
\] 

We say that a sequence $(x_n) \in X$ is strongly (resp.~weakly) \textbf{{\boldmath$p$}-summable} if 
\[ 
\| (x_n)\|_p = \left( \sum_n \| x_n\|^p\right)^{1/p} < \infty
\]
(resp.~$(x^*(x_n)) \in \ell^p$ for each $x^* \in X^*$). 
The set $\ell^p_s(X)$ 
of strongly $p$-summable sequences is a Banach space and identified with 
the completion (denoted by $\ell^p\otimes_pX$) of 
$\ell^p\otimes X$ with respect to the intermediate cross norm $\| \cdot\|_p$. 

\begin{Example}~ 
\begin{enumerate}
\item
For $p=1$, the equality $\ell^1_s(X) = \ell^1\overline{\otimes} X$ holds 
because $\| \cdot\|_1$ on $\ell^1\otimes X$ coincides with the maximal cross norm. 
In fact, for $\varphi \in \sB(\ell^1,X)$, 
\[ 
|\varphi(\sum_n \delta_n \otimes x_n)| = |\sum_n \varphi(\delta_n,x_n)|
\leq \sum_n \| \varphi\|\, \| x_n\| 
= \| \varphi\|\, \| (x_n)\|_1,  
\] 
which shows that the maximal cross norm is majorized by the cross norm $\| \cdot\|_1$. 
\item
Let $p = 2$ and consider the case $X = \ell^2$. 
Since the norm $\| \cdot\|_2$ on $\ell^2_s(X)$ corresponds to the Hilbert-Schmidt norm on 
linear maps $\ell^2 \to X^*$, we have $\ell^2\otimes_2 X = \sC_2(\ell^2)$, 
which is different from $\ell^2\overline{\otimes} X = \sC_1(\ell^2)$. 
\end{enumerate}
\end{Example}

Let $\ell^p_w(X)$ be the vector space of weakly $p$-summable sequences. 
Note that each $(x_n) \in \ell^p_w(X)$ gives rise to 
a linear map $X^* \ni x^* \mapsto (x^*(x_n)) \in \ell^p$, which is bounded due to the closed graph 
theorem, and any bounded linear map of $X^*$ into $\ell^p$ arises this way. 
Thus $\ell^p_w(X)$ is identified with $\sL(X^*,\ell^p)$ so that $\ell^p\underline{\otimes}X$ 
is a closed subspace of $\ell^p_w(X)$ with their norms given by the common formula 
\[ 
\|(x_n)\|_{p,w} = \sup \{ (\sum_n |x^*(x_n)|^p)^{1/p}; x^* \in X^*_1\}. 
\]
Moreover, the obvious inclusion $\ell^p_s(X) \subset \ell^p_w(X) = \sL(X^*,\ell^p)$ 
is contractive and $\ell^p_s(X) \subset \ell^p\underline{\otimes} X \subset \ell^p_w(X)$. 

Note that the assignments $\ell^p_s(X)$ and $\ell^p_w(X)$ are functorial 
in the category of Banach spaces: For a $\phi: \sL(X,Y)$, the correspondence 
$(x_n) \mapsto (\phi x_n)$ gives rise to $\ell^p_s(\phi) \in \sL(\ell^p_s(X),\ell^p_s(Y))$ 
and $\ell^p_w(\phi) \in \sL(\ell^p_w(X),\ell^p_w(Y))$ 
in a functorial way together with inequalities 
$\| \ell^p_s(\phi)\| \leq \| \phi\|$ and $\| \ell^p_w(\phi)\| \leq \| \phi\|$. 

\begin{Definition}
For each $1 \leq p \leq \infty$, 
extend a bounded linear map $\phi: V \to W$ between Banach spaces to 
$\ell^p_s(\phi): \ell^p_s(V) \to \ell^p_s(W)$ by $(v_n) \mapsto (\phi(v_n))$ 
so that $\| \ell^p_s(\phi) \| = \| \phi\|$. 
A bounded linear map $\phi: V \to W$ is said to be \textbf{{\boldmath$p$}-summing} if 
$\ell^p_s(\phi)$  splits through $\ell^p\underline{\otimes} V$; 
$\phi^p_s(\phi)$ is expressed as a composition of the contractive embedding 
$\ell^p_s(V) \to \ell^p\underline{\otimes} V \subset \ell^p_w(V)$ 
and a bounded linear map ${\ell^p}(\phi): \ell^p\underline{\otimes} V \to \ell^p_s(W)$. 
\end{Definition}

\begin{Proposition}
The set $\sL^p(V,W)$ of $p$-summing linear maps is a Banach space with respect to 
the norm $\|\ell^p(\phi)\|$. 
\end{Proposition}

\begin{proof}
In fact, suppose that a sequence $(\ell^p(\phi_\alpha))$ converges to 
$\Phi$ in $\sL(\ell^p\underline{\otimes}V,\ell^p_s(W))$. 
Then $\langle \delta_n, \ell^p(\phi_\alpha)(\delta_n\otimes v) \rangle = \phi_\alpha(v)$ 
is convergent in $W$ for any $v \in V$ and any $n \geq 1$. 
If we set $\phi(v) = \lim \phi_\alpha(v)$, $\phi$ is bounded with $\| \phi\| \leq \sup \{ \|\phi\|_\alpha\}$ 
finite by Banach-Steinhauss theorem. 
Since $\langle\delta_n,\Phi(\delta\otimes v) \rangle = \phi(v)$ irrelevant of $n \geq 1$, 
$\ell^p_s(\phi) = \Phi$ on a dense linear subspace $\ell_0\otimes V$ of $\ell^p\underline{\otimes} V$, 
whence $\phi$ is $p$-summing and $\ell^p(\phi) = \Phi$. 
\end{proof}

The $p$-summing norm 
$\| \ell^p(\phi)\|$ is the infimum of $\rho>0$ satisfying 
\[
\left(\sum_{j=1}^n \| \phi(v_j)\|^p\right)^{1/p} 
\leq \rho \sup\left\{ (\sum |v^*(v_j)|^p)^{1/p}; v^* \in V^*_1\right\} 
\] 
for any finite sequence $\{ v_j\}_{1 \leq j \leq n}$ in $V$. 

\begin{Example}
When $V$ and $W$ are Hilbert spaces, $\sL^2(V,W) = \sC_2(V,W)$ 
so that $\| \ell^2(\phi)\|$ is equal to the Hilbert-Schmidt norm of $\phi: V \to W$. 

In fact, the condition on $\rho$ takes the form 
\[ 
\sum_j \| \phi(v_j)\|^2 \leq \rho^2 \sup \{ \sum |(v|v_j)|^2; v \in V_1 \}.  
\] 
In terms of the positive operator $h = \sum_j |v_j)(v_j|$, we can write
$\sum_j \| \phi v_j\|^2 = \text{tr}(\phi^*\phi h)$ and $\sum_j |(v|v_j)|^2 = (v|hv)$. 
Thus $\sup\{ \sum_j |(v|v_j)|^2; v \in V_1\} = \| h\|$ and the ineqaulity for $\rho$ becomes 
$\text{tr}(\phi^*\phi h) \leq \rho^2 \| h\|$. 
Since any positive finite rank operator $h$ is of the form $\sum_j |v_j)(v_j|$, this is further 
equivalent to $\text{tr}(\phi^*\phi h) \leq \rho^2$ for any 
finite rank operator $h$ satisfying $0 \leq h \leq 1$. Then, by maximizing on $h$, the condition on 
$\rho$ is boiled down to $\text{tr}(\phi^*\phi) \leq \rho^2$. 
\end{Example}

\begin{Proposition} Let $X,X',Y,Y'$ be Banach spaces. 
\begin{enumerate}
\item 
Let $a \in \sL(X',X)$ and $b \in \sL(Y,Y')$. Then, for $\phi \in \sL^p(X,Y)$, 
$b\phi a \in \sL^p(X',Y')$ and $\| \ell^p(b\phi a)\| \leq \| b\|\,\| \ell^p(\phi)\|\,\| a\|$. 
\item 
For $1 \leq p \leq q < \infty$, $\sL^p(X,Y) \subset \sL^q(X,Y)$ so that  
$\| \ell^q(\phi)\| \leq \| \ell^p(\phi)\|$ for $\phi \in \sL^p(X,Y)$. 
\end{enumerate}
\end{Proposition}

\begin{proof} 
(i) follows from $\| \ell^p_s(a)\| \leq \|a\|$ and $\| \ell^p_w(b)\| \leq \| b\|$. 

(ii) Let $\phi \in \sL^p(X,Y)$. Given finite sequences $\{ x_j\}$ in $X$ and $\{ \lambda_j\}$ 
of scalars, the H\"older's inequality for the exponents $1/p = 1/q + 1/r$ is applied to obtain 
\[ 
\| (\lambda_j \phi x_j)\|_p \leq \| \ell^p(\phi)\|\, \| (\lambda_j x_j)\|_{p,w} 
\leq \| \ell^p(\phi)\|\, \| (\lambda_j)\|_r\, \|( x_j)\|_{q,w}.  
\] 
If we choose $\lambda_j$ so that 
\[ 
\left( \| (\lambda_j \phi x_j)\|_p\right)^p = \sum |\lambda_j|^p\ \| \phi x_j\|^p 
= \sum \| \phi x_j\|^q, 
\] 
i.e., $\lambda_j = \| \phi x_j\|^{q/r}$, then we have 
$\| (\phi x_j)\|_q = \| (\lambda_j \phi x_j)\|_p/\| (\lambda_j)\|_r$,  
which is combined with above inequality to get the inequality 
$\|(\phi x_j)\|_q \leq \| \ell^p(\phi)\|\, \| (x_j)\|_{q,w}$. 
\end{proof} 

\begin{Proposition}
The inclusion map $\ell^1 \to \ell^2$ is $1$-summing. 
\end{Proposition} 

\begin{proof}
First recall the lower Khintchine's inequality of the following form: There exists $C>0$ such that, 
for $a = (a_k) \in \ell^1 \subset \ell^2$, 
\[ 
\sqrt{\sum_k |a_k|^2} \leq C \int_0^1 \left| \sum_k a_k r_k(t) \right|\, dt. 
\]
Here $\{ r_k(t)\}_{k \geq 1}$ denotes the Rademacher functions. 

The Khintchine's inequality is then applied to $x_1, \cdots, x_n \in \ell^1$ to get 
\[ 
\sum_j \| x_j\|_2 \leq C \int_0^1 \sum_j \left| \sum_k x_{j,k} r_k(t) \right|\, dt. 
\] 
Since $(r_k(t))$ belongs to the unit ball of $\ell^\infty = (\ell^1)^*$ for $0 \leq t \leq 1$, 
the integrand is estimated as 
\[ 
\sum_j \left| \sum_k x_{j,k} r_k(t) \right| \leq 
\sup \{ \sum_j |x^*(x_j)|; x^* \in \ell^\infty, \| x^*\|_\infty \leq 1\} 
= \| (x_j)\|_{1,w}
\] 
and we finally obtain $\| (x_j)\|_2 \leq C \| (x_j)\|_{1,w}$, i.e., 
$\| \ell^1(\ell^1 \subset \ell^2)\| \leq C$. 
\end{proof}

\begin{Corollary}
For Hilbert spaces $X$ and $Y$, $\sL^p(X,Y) = \sL^2(X,Y)$ for $1 \leq p \leq 2$. 
\end{Corollary}

\begin{proof}
We need to show that every $\phi \in \sL^2(X,Y)$ belongs to $\sL^1(X,Y)$. 
Since $\phi$ is then in the Hilbert-Schmidt class, 
the spectral decomposition followed by polar decomposition of $\phi$ reduces the problem to 
the case $X=Y=\ell^2$ and $\phi$ is a multiplication operator by a sequence $(\phi_n) \in \ell^2$. 
Then the image of $\phi$ is included in $\ell^1$ so that $\phi: \ell^2 \to \ell^1$ is bounded: 
$\|(\phi_n \xi_n)\|_1 \leq \| (\phi_n)\|_2\, \|(\xi_n)\|_2$. Thus $\phi$ is realized 
as a bounded linear map $\ell^2 \to \ell^1$ followed by the inclusion map $\ell^1 \to \ell^2$ and 
we see that $\|\ell^1(\phi)\| \leq \|\ell^1(\ell^1 \subset \ell^2)\|\, \| \phi:\ell^2 \to \ell^1\| < \infty$. 
\end{proof} 

Given a cross norm $\|\cdot\|$ on $X\otimes Y$, consider a $\|\cdot\|$-bounded linear functional 
$\varphi: X\otimes Y \to \C$. Since elementary tensors are total in $X\otimes Y$ with respect to 
$\| \cdot\|$, the restriction $\varphi \mapsto \varphi|_{X\times Y} = \varphi_{X\times Y}$ is injective and 
$\varphi_{X\times Y}$ belongs to $\sB(X,Y)$ in view of $\| \varphi\|_{\sB(X,Y)} \leq \| \varphi\|$. 
Thus $(X\otimes Y)^*_{\| \cdot\|}$ is continuously embedded into $\sB(X,Y) = \cL(X,Y^*)$. 
We shall here give an expression of $\| \varphi\|_l$ in terms of the associated linear map 
$\phi: X \to Y^*$ defined by $\langle \phi(x),y\rangle = \varphi(x\otimes y)$.  

\begin{Theorem}[Swartz1976] A linear functional $\varphi$ on $X\otimes Y$ 
is $\|\cdot\|_l$-continuous if and only if the associated operator $\phi: X \to Y^*$ is $1$-summing. 
Moreover, we have $\| \varphi\|_l = \| \ell^1(\phi)\|$. 
\end{Theorem}

\begin{proof} We first rewrite the definition of $\| \cdot\|_l$ slightly. 
For $z = \sum_j x_j\otimes y_j \in X\otimes Y$, we have 
\begin{align*}
\sup \{ \left\|\sum \alpha_j \| y_j\| x_j \right\|; |\alpha_j| \leq 1 \}
&= \sup \{ \left|\sum \alpha_j x^*(\| y_j\| x_j)\right|; x^* \in X^*_1, |\alpha_j| \leq 1 \}\\
&= \sup\{ \sum |x^*(\| y_j\| x_j)|; x^* \in X^*_1 \} 
\end{align*}
and hence  
\[ 
\| z\|_l = \inf \sup\{ \sum \| y_j\|\,|x^*(x_j)|; x^* \in X^*_1 \}, 
\]
where the infimum is taken over possible expressions $\sum_j x_j\otimes y_j$ of $z \in X\otimes Y$. 

Suppose that $\| \varphi\|_l < \infty$. 
Given a finite sequence $\{ x_j\}_{1 \leq j \leq n}$ in $X$ and $\epsilon>0$, 
choose $y_j \in Y_1$ so that $\| \phi(x_j)\| - \epsilon \leq \langle \phi(x_j),y_j\rangle$ 
for $1 \leq j \leq n$ and 
set $z = \sum_j x_j\otimes y_j \in X\otimes Y$. Then 
\begin{align*}
\sum_{j=1}^n \| \phi(x_j)\| - n \epsilon 
&\leq \sum \langle \phi(x_j),y_j\rangle 
= \varphi(z) \leq \| \varphi\|_l \| z\|_l\\ 
&\leq \| \varphi\|_l \sup\{ \sum |x^*(\| y_j\| x_j)|; x^* \in X^*_1 \} 
\leq \| \varphi\|_l \sup\{ \sum |x^*(x_j)|; x^* \in X^*_1 \}.   
\end{align*}
Since the first and the last expressions are independent of the choice of $y_j$, 
we can take the limit $\epsilon \to 0$ to have  
$\sum_j \| \phi(x_j)\| \leq \| \varphi\|_l \| \sum_j \delta_j\otimes x_j\|_{\underline{\otimes}}$, 
which shows that $\phi$ is $1$-summing and $\| \varphi\|_l \leq \| \ell^1(\phi)\|$. 

To get the reverse inequality, assume that $\phi$ is $1$-summing and, 
for $z = \sum x_j\otimes y_j \in X\otimes Y$, estimate as 
\[ 
|\varphi(z)| = | \sum \langle \phi(x_j),y_j\rangle| 
\leq \sum \| \phi(x_j)\|\, \| y_j\| 
= \sum \| \phi(\| y_j\| x_j)\|
\leq \| \ell^1(\phi)\| \sup\{ \sum |x^*(\| y_j\| x_j)|; x^* \in X^*_1 \}. 
\]
By taking infimum over possible expressions $\sum_j x_j\otimes y_j = z$, 
we get $|\varphi(z)| \leq \|\ell^1(\phi)\|\, \|z\|_l$, i.e., $\| \varphi\|_l \leq \| \ell^1(\phi)\|$. 
\end{proof}

\begin{Corollary}\label{corollary}
For Hilbert spaces $V$ and $W$, the Banach space space $V\otimes_lW$ is topologically equal to 
the Hilbert space $V\otimes W$. 
Consequently, the tensor product semi-measure $\varphi\otimes \psi: \sA\otimes \sB \to V\otimes W$ 
is lifted to a measure $\sA\otimes_\sigma\sB \to V\otimes W$. 
\end{Corollary}

\begin{proof}
Since $(V\otimes_lW)^*$ is hilbertian, $V\otimes_lW$ itself is hilbertitan as a closed linear subspace 
of a hilbertian $(V\otimes_l W)^{**}$. Then $V\otimes_lW$ is topologically equal to 
$(V\otimes_lW)^{**}$, which is nothing but the ordinary Hilbert space tensor product $V\otimes W$ 
as the dual of the space of Hilbert-Schmidt operators. 
\end{proof}

Now we can state and prove a theorem of 
our main concern in this notes. The following is mostly contained 
in Swartz1976, but not whole. 
Also, relevant ingredients for the proof is scattered over 
variours papers by many researchers. So, we shall try here to show a minimal route for access. 

A semi-measure $\phi$ on a Boolean algebra $\sB$ 
with values in a Hilbert space is said to be \textbf{orthogonal} if 
$\phi(A) \perp \phi(B)$ whenever $A \cap B = \emptyset$ in $\sB$. 
The semi-variation of an orthogonal semi-measure $\phi$ takes an especially simple form: 
$|\phi|(A) = \| \phi(A)\|$ for $A \in \sB$, which is not additive unless $\phi$ is supported by 
an atomic set in $\sB$ but always bounded with $\| \phi\| = \| \phi(S)\|$. 
As a result of boundedness, $\phi$ is squeezing. 
In fact, if $\bigsqcup A_n$ and $\|\phi(A_n)\| \geq \delta$ for all $n \geq 1$, then 
$\|\phi(S)\| \geq \|\phi(\sqcup_{n=1}^N A_n)\| \geq \sqrt{\sum_{n=1}^N \|\phi(A_n)\|^2} 
\geq \sqrt{N\delta}$ can increase unlimitedly. 

\begin{Lemma}\label{lemma}
Let $\sH$ be a finite-dimensional Hilbert space and $T$ be a positive operator on $\sH$. 
Then we can find orthogonal measures $\xi,\eta: 2^\N \to \sH$ satisfying 
$\|\xi\| = 1 = \|\eta\|$ and $\| (\xi|T\eta) \| = \| T\|_2$. 

Here the complex semi-measure $(\xi|T\eta)$ on $2^\N\otimes 2^\N \subset 2^{\N\times \N}$ is specified by 
$(\xi|T\eta)(A\times B) = (\xi(A)|T\eta(B))$ for $A, B \in 2^\N$ and $\| T\|_2$ denotes 
the Hilbert-Schmidt norm of $T$. 
\end{Lemma}

\begin{proof}
Since the Boolean algebra $2^\N\otimes 2^\N$ is atomic, 
\[ 
\|(\xi|T\eta)\| = \sum_{j,k} |(\xi_j|T\eta_k)|, 
\quad \xi_j = \xi(\{ j\}), \eta_k = \eta(\{ k\}). 
\] 
We now restrict $\xi$ and $\eta$ to be supported by the set $\{ 1,2,\dots,\dim\sH\} \subset \N$ 
and choose orthonormal bases $\{ e_j\}$ and $\{ f_j\}$ in $\sH$ so that 
$\xi_j = \|\xi_j\| e_j$ and $\eta_j = \| \eta_j\| f_j$ for $1 \leq j \leq \dim \sH$. 
Then, under the condition $\| \xi\| = \| \eta\| = 1$, orthogonal measures $\xi$ and $\eta$ 
are compactly parametrized and the problem is reduced to showing that $\| T\|_2$ is realized as  
\[ 
\max \{ \sum_{j,k} \| \xi_j\|\,|(e_j|T f_k)|\,\| \eta_k\|; 
\sum_j \| \xi_j\|^2 = 1 = \sum_k \| \eta_k\|^2\}
= \| [e|T f] \|
\]
for some orthonomal bases $\{ e_j\}$, $\{ f_k\}$ of $\sH$. 
Here $\| [e|Tf] \|$ denotes the operator norm of the matrix $[e|Tf] = (|(e_j|Tf_k)|)$. 

Let $T = \sum_{1 \leq j \leq \dim\sH} t_j |g_j)(g_j|$ be a spectral expression 
with $\{ g_j\}$ an orthonormal basis. 
If we set $f_j = g_j$, then $|(e_j|Tf_k)| = |(e_j|g_k)| t_k$ and $(e_j|g_k)$ can be any unitary matrix, 
which allows us to choose $(e_j|g_k) = e^{2\pi ijk/n}/\sqrt{n}$ and get 
$\| [e|Tg]\| = \sqrt{t_1^2 + \dots + t_n^2} = \| T\|_2$: 
\[ 
[e|Tg] = \frac{1}{\sqrt{n}} 
\begin{pmatrix}
1 & \dots & 1\\
\vdots & \ddots & \vdots\\
1 & \dots & 1
\end{pmatrix}
\begin{pmatrix}
t_1 & & 0\\
 & \ddots & \\
0 &  & t_n
\end{pmatrix} 
= \frac{1}{\sqrt{n}} 
\begin{pmatrix}
1\\
\vdots\\
1
\end{pmatrix} 
\begin{pmatrix}
t_1 & \cdots & t_n
\end{pmatrix}
\]
with the norm of the last matrix equal to $\| (t_1,\cdots,t_n)\| = \sqrt{t_1^2 + \cdots + t_n^2}$. 
\end{proof}

\begin{Remark}
For a real Hilbert space of $\dim \sH = 2^m$, the conclusion of Lemma remains true by 
taking $(e_j|e_k)$ to be the $m$-times tensor product of two-dimensional reflection (or rotation) 
matrix by an angle $\pi/4$ as utilized in [Dudley-Pakula1972].
\end{Remark} 

\begin{Theorem} 
Let $T: \sH \to \sK$ be a bounded linear map between Hilbert spaces. 
Then the following conditions are equivalent. 
\begin{enumerate}
\item 
Given an $\sH$-valued measure $\xi$ on $\sA$ and a $\sK$-valued measure $\eta$ on $\sB$, 
the semi-measure $(\xi|T\eta)$ on $\sA\otimes \sB$ is extended to a complex measure on 
$\sA\otimes_\sigma\sB$. 
\item 
Given a $\overline{T\sH}$-valued orthogonal measure $\xi$ on $2^\N$ and 
a $\ker(T)^\perp$-valued orthogonal measure $\eta$ on $2^\N$, the semi-measure $(\xi|T\eta)$ on 
$2^\N \otimes 2^\N$ is bounded. 
\item 
$T$ is in the Hilbert-Schmidt class. 
\end{enumerate}
\end{Theorem} 

\begin{proof} 
(iii) $\Longrightarrow$ (i) has been already established (Corollary~\ref{corollary}), 
whereas (i) $\Longrightarrow$ (ii) is due to Proposition~\ref{variations} (iv). 

So we focus on (ii) $\Longrightarrow$ (iii). 
For this, we first notice that, for isometries $U: \sH \to \sH'$ and $V: \sK \to \sK'$, 
operators $T$ and $VTU^*$ share the validity of (ii) in common, 
so we may assume that $\sH = \sK$ and $T\geq 0$ with a dense range by polar decomposition. 
Let $E$ be a projection in $\sH$. Then $ETE$ is injective on $E\sH$ 
($ETE\xi = 0$ implies $T^{1/2}E\xi = 0$ and threfore $E\xi = 0$) and, if $T$ has the property (ii), 
so does the reduced operator $ETE$ on $E\sH$. 

We now assume that the positive operator $T$ with a trivial kernel 
is not in the Hilbert-Schmidt class. 
Then we can find a decomposition of the identity operator 
into a sequence of mutually orthogonal infinite-dimensional projections $\{ E_n\}$ so that 
$E_nT = TE_n$ and $E_nTE_n$ is not in the Hilbert-Schmidt class. 
(If $\sigma(T)$ is not a finite set, we can take $E_n$ to be spectral projections of $T$ 
according to a partition of $\sigma(T)$ by countably many subsets. Otherwise, 
$T$ has an eigenvalue $t>0$ of infinite multiplicity and take a decomposition $[T=t] = \sum E_n$ 
with the spectral projection $[T \not= t]$ added to, say, $E_1$.)

With these preparatory discussions, we extract the essence of [Dudley-Pakula1972] as follows. 
Let $(\epsilon_n) \in \ell^2$ with $\epsilon_n > 0$ be any auxiliary sequence. 
Since $(E_nTE_n)^2$ is not in the trace class, 
$\sum_{i,j \in I_n} (e_{n,i}|Te_{n,j})(e_{n,j}|Te_{n,i}) = \infty$ for an ONB $\{ e_{n,i}\}_{i \in I_n}$ 
of $E_n\sH$ and we can choose a finite subset 
$F_n \subset I_n$ so that $\sum_{i,j \in F_n} (e_{n,i}|Te_{n,j})(e_{n,j}|Te_{n,i}) \geq 1/\epsilon_n^4$. 
Let $P_n$ be the projection to $\sum_{i \in F_n} \C e_{n,i}$. 
Then the finite-dimensional $P_n \leq E_n$ satisfies $\| P_nTP_n\|_2 \geq 1/\epsilon_n^2$ and 
we apply Lemma~\ref{lemma} to find measures $\xi_n,\eta_n: 2^\N \to P_n\sH$ fulfilling 
$\| \xi_n\| = \| \eta_n \| = \epsilon_n$ and 
$\| (\xi_n|T\eta_n)\| = \epsilon_n^2 \| P_nTP_n\|_2 \geq 1$ for each $n \geq 1$. 
Introduce orthogonal measures $\xi,\eta: 2^{\N\times \N} \to \sH$ by 
$\xi(A) = \sum_n \xi_n(A_n)$ for $A \in 2^{\N\times \N}$ with $A_n = \{ k \in \N; (k,n) \in A\}$ 
so that $\| \xi\|^2 = \sum_n \epsilon_n^2 < \infty$, and similarly for $\eta$. 
\begin{align*}
\| (\xi|T\eta)\| 
&= \sum_{k,l,m,n} |(\xi(k,m)|T\eta(l,n))| 
= \sum_{k,l,m,n} |(\xi_m(k)|T\eta_n(l))|\\ 
&= \sum_{k,l,m,n} |(\xi_m(k)|E_mTE_n\eta_n(l))| 
= \sum_{k,l,n} |(\xi_n(k)|E_nTE_n\eta_n(l))|\\ 
&= \sum_n \|(\xi_n|T\eta_n)\|, 
\end{align*}
which diverges because of $\| (\xi_n|T\eta_n) \| \geq 1$ and the property (ii) fails to 
be satisfied by $T$. 
\end{proof}

\appendix
\section{Khintchine's Inequalities}
The following is based on  \cite[Appendix C]{Gr2014}. 

Let $s_n$ be an independent sequence of random variables with the property $\mu(s_n = \pm 1) = 1/2$ 
for every $n \geq 1$. 

\begin{Example}
Let $\Omega = \prod_1^\infty \{ \pm 1\}$ with the product probability measure $\mu$ of equal weights. 
The random variable $s_n$ is then obtained by extracting the $n$-th component of $\omega \in \Omega$. 

If we apply the binary expansion to the interval $[0,1]$, 
the Lebesgue measure on $[0,1]$ is identified with the product measure of equal weights
on $\prod_1^\infty \{ 0, 1\}$, which is further identified with $\prod_1^\infty \{ \pm 1\}$ 
by the correspondence $(1,-1) \leftrightarrow (0,1)$. 
The random variable $s_n$ is then identified 
with a measurable function $r_n$ on $[0,1]$. Its explicit form is the following: 
Let a periodic function $r_1:\R \to \{ \pm 1\}$ of period $1$ be defined by $r_1(t) = 1$ ($0\leq t < 1/2$) 
$r_1 = -1$ ($1/2 \leq t < 1$) and set $r_n(t) = r_1(2^{n-1}t)$. 
The functions $r_n$ are referred to as Rademacher functions. 
\end{Example}

For $1 \leq p < \infty$, 
consider a linear map $K_p: \ell^1 \ni a = (a_n) \mapsto K_pa = \sum_n a_n s_n \in L^p(\Omega,\mu)$. 
Due to the oscillating sum effect, the obvious boundedness of this map can be improved 
so that it splits through the inclusion $\ell^1 \subset \ell^2$, i.e., 
$C_p = \sup\{ \| K_p(a)\|_p;  \| a\|_2 = 1\}$ can be finite. Khintchine's inequalities assert 
more strongly that the closure of $K_p\ell^1$ in $L^p(\Omega,\mu)$ is topologically isomorphic 
to $\ell^2$. 

\begin{Example}~ 
\begin{enumerate}
\item 
For the case $p=2$, 
\[ 
\| K_2(a)\|_2^2 = \sum_{j,k} \overline{a_j} a_k \int_\Omega s_j(\omega) s_k(\omega)\, \mu(d\omega) 
= \sum_n |a_n|^2. 
\] 
\item 
For $1 \leq p < 2$, let $q > 2$ be defined by $1/p = 1/2 + 1/q$. By H\"older's inequality, 
$\| f\|_p \leq \| 1\|_q \| f\|_2 = \| f\|_2$ for $f \in L^p(\Omega,\mu)$ and then, by duality, 
$\| f\|_2 \leq \| f\|_{p'}$ for $f \in L^{p'}(\Omega,\mu)$, where $p' > 2$ is the dual exponent of $p$. 
Now we observe that 
$\| K_pa\|_p \leq \| K_2a\|_2 = \| a\|_2$ for $1 \leq p \leq 2$ and 
$\| a\|_2 = \| K_2a\|_2 \leq \| K_pa\|_p$ for $2 \leq p < \infty$. 
\end{enumerate}
\end{Example} 

\begin{Theorem}[Khintchine's inequalities]~
For each $1 \leq p < \infty$, let $C_p > 0$ be the best constant of the following inequality 
on a sequence $(a_n) \in \ell^1$ of complex numbers. 
  \begin{enumerate}
  \item 
For $2 \leq p < \infty$, 
\[ 
\left( \int_\Omega |\sum_n a_n s_n(\omega)|^p\, \mu(d\omega) \right)^{1/p} 
\leq C_p \sqrt{\sum_n |a_n|^2}. 
\]
\item 
For $1 \leq p \leq 2$, 
\[ 
\sqrt{\sum_n |a_n|^2} \leq C_p \left( \int_\Omega |\sum_n a_n s_n(\omega)|^p\, \mu(d\omega) \right)^{1/p}. 
\] 
  \end{enumerate} 
Then $C_p \leq 2p^{1/p} \Gamma(p/2)^{1/p}$ for $p > 2$ and 
$C_p \leq C_{4-p}^{4/p -1}$ for $1 \leq p < 2$. In particular, we have 
\[ 
\sqrt{\sum_n |a_n|^2} \leq 12\sqrt{\pi} \int_\Omega |\sum_n a_n s_n(\omega)|\, \mu(d\omega). 
\] 
\end{Theorem} 

\begin{proof}
We first show that (ii) is a consequence of (i): Let $1 \leq p \leq 2$. Then, 
\begin{align*} 
(a|a) &= \int |\sum_n a_n s_n(\omega)|^2\, \mu(d\omega)\\ 
&= \int |\sum_n a_n s_n(\omega)|^{p/2} |\sum_n a_n s_n(\omega)|^{2-p/2}\, \mu(d\omega)\\ 
&\leq \left( \int |\sum_n a_n s_n(\omega)|^p \right)^{1/2} 
\left( \int |\sum_n a_n s_n(\omega)|^{4-p} \right)^{1/2}\, \mu(d\omega)\\ 
&\leq \left( \int |\sum_n a_n s_n(\omega)|^p \right)^{1/2} C_{4-p}^{2 - p/2} 
\| a\|_2^{2 - p/2}, 
\end{align*}
whence 
\[ 
\| a\|_2^{p/2} \leq C_{4-p}^{2 - p/2} \| \sum_n a_n s_n \|_p^{p/2},  
\] 
i.e., $\| a\|_2 \leq C_{4-p}^{4/p - 1} \| \sum_n a_n s_n\|_p$. 

Now let $p \geq 2$ and we focus on (i). For the moment, we assume $a_n \in \R$. 
In view of the equality
\[ 
\int_\Omega |f(\omega)|^p\, \mu(d\omega) 
= \int_\Omega p\int_0^{|f(\omega)|} t^{p-1}\, dt\, \mu(d\omega)
= p \int_{0 < t < |f(\omega)|} t^{p-1} \mu(d\omega) dt 
= p \int_0^\infty t^{p-1} \mu(|f| > t)\, dt 
\] 
for a measurable function $f$ on $\Omega$, we try to capture how $\mu(|\sum_a a_n s_n| > t)$ 
behaves as $t$ increases. 
To this end, we estimate $\int_\Omega e^{t|\sum_n a_n s_n(\omega)|}\, \mu(d\omega)$ in two ways: 
The first one is the obvious lower bound and given by 
\[ 
\int_\Omega e^{t |\sum a_n s_n(\omega)|}\, \mu(d\omega) \geq e^{t^2} \mu(|\sum_n a_n s_n| > t).  
\] 
The second one is about an upper bound, for which we use the inequality $e^x + e^{-x} \leq 2 e^{x^2/2}$ 
(compare Taylor coefficients) to get 
\[ 
\int_\Omega e^{\pm t \sum a_n s_n(\omega)}\, \mu(d\omega) 
= \prod_n \int_\Omega e^{\pm t a_n s_n(\omega)}\, \mu(d\omega) 
= \prod_n \frac{e^{t a_n} + e^{-t a_n}}{2} \leq e^{t^2\sum_n a_n^2/2} = e^{t^2 (a|a)/2}, 
\] 
and then 
\[ 
\int_\Omega e^{t |\sum a_n s_n(\omega)|}\, \mu(d\omega) 
\leq \int_\Omega (e^{t\sum a_n s_n(\omega)} + e^{-t\sum a_n s_n(\omega)})\, \mu(d\omega) 
\leq 2e^{t^2 (a|a)/2}. 
\] 
Combinig these, we obtain the desired estimate $\mu(|\sum_n a_n s_n| > t) \leq 2 e^{-t^2/2(a|a)}$, 
which is used in the $t$-integral expression for $\| K_p(a)\|_p^p$ to have 
\[ 
\| K_p(a)\|_p^p \leq 2p  \int_0^\infty t^{p-1} e^{-t^2/2(a|a)}\, dt 
= p 2^{p/2} (a|a)^{p/2} \Gamma(p/2),  
\] 
i.e., $\|K_p(a)\|_p \leq \sqrt{2} p^{1/p} \Gamma(p/2)^{1/p} \| a\|_2$ for a real $(a_n) \in \ell^1$. 

A complex sequence $c_n = a_n + ib_n$ is handled with help of 
Minkowski inequality and the usual estimate $\| a\|_2 + \| b\|_2 \leq \sqrt{2}\| c\|_2$ as 
\[ 
\| K_p(c)\|_p \leq \| K_p(a)\|_p + \| K_p(b)\|_p 
\leq \sqrt{2} p^{1/p} \Gamma(p/2)^{1/p} (\| a\|_2 + \| b\|_2) 
\leq 2 p^{1/p} \Gamma(p/2)^{1/p} \| c\|_2,  
\] 
showing $C_p \leq 2p^{1/p} \Gamma(p/2)^{1/p}$ for $p \geq 2$. 
\end{proof}


\bigskip
\noindent

\begin{thebibliography}{20}
\bibitem{D1985}
J.~Diestel, 
An introduction to the theory of absolutely p-summing operators between Banach spaces, 
Proceedings of the Centre for Mathematical Analysis, v.~9(1985), 1--26.
\bibitem{DU1977}
J.~Diesel and J.J.~Uhl, Vector Measures, AMS, 1977. 
\bibitem{DP1972}
R.M.~Dudley and L.~Pakula, A counter-example on the inner product of measures, 
Indiana Univ.~Math.~J., 21(1972), 843--845. 
\bibitem{GV1997}
J.C.~Garcia-Vazquez, Products of vector measures, 
J.~Math.~Anal.~Appl., 206(1997), 25--41. 
\bibitem{GL1965}
Jesús G̀il de Lamadrid, Measures and Tensors, 
Trans.~Amer.~Math.~Soc., 
114(1965), 98-121. 
\bibitem{Gr2014}
L.~Grafakos, Classical Fourier Analysis, Springer, 2014. 
\bibitem{Kl1973}
I.~Kluv\'anek, On the product of vector measures, 
J.~Austral.~Math.~Soc., 15(1973), 22--26. 
\bibitem{Ri1999}
Werner Ricker, Operator algebras generated by commuting projections, LNM 1711(1999). 
\bibitem{Ry2002}
Raymond A.~Ryan, Introduction to Tensor Products of Banach Spaces, 
Springer, 2002. 
\bibitem{S1976}
C.~Swartz, Tensor products of $\ell^2$-valued measures, 
J.~Austral.~Math.~Soc., 21(1976), 241-246.
\bibitem{TM1973}
Don H.~Tucker and Hugh B.~Maynard,
Vector and Operator Valued Measures and Applications, Elsevier, 1973. 
\bibitem{Y2007} 
S.~Yamagami, Yet Another Lebesgue Integration, 2007.\\ 
https://www.math.nagoya-u.ac.jp/\~{}yamagami/teaching/topics/yali.pdf
\end{thebibliography}
\end{document}